\documentclass{article}
\usepackage{amssymb}
\usepackage{amsfonts}
\usepackage{amsmath}
\usepackage{fullpage}
\usepackage{graphicx}
\usepackage{mathrsfs}

\newtheorem{theorem}{Theorem}

\newtheorem{corollary}[theorem]{Corollary}

\newtheorem{definition}[theorem]{Definition}

\newtheorem{lemma}[theorem]{Lemma}

\newtheorem{proposition}[theorem]{Proposition}

\newenvironment{proof}[1][Proof]{\noindent{\sc#1} }{\nopagebreak\hfill \rule{0.5em}{0.5em}\\ }
\begin{document}

\title{Tail generating functions for 
Markov branching processes}
\author{Serik Sagitov\thanks{%
Chalmers University and University of Gothenburg, 412 96 Gothenburg, Sweden. Email address:
serik@chalmers.se. Supported by the Swedish Research Council grant
621-2010-5623.}
}
\maketitle

\begin{abstract}
We give a concise self-contained presentation of known and new limit theorems for the one-type Markov branching processes with continuous time. The new streamlined proofs are based on what we call, the tail generating function approach.  Our analysis focuses on the singularity points of the master integral equation for the probability generating functions of the current population size.
\end{abstract}

\section{Introduction}
The traditional way of presenting the theory of branching processes is to start with the discrete time Galton-Watson processes and then proceed with the continuous time Markov branching processes. The majority of the monographs on the theory of branching processes follow this order \cite{AH, AN, HJV, Ha, J, KA, M} with the exception of \cite{Se}, where the presentation of the  Galton-Watson processes and Markov branching processes is given in parallel. A major reason for this choice is the fact that the class of Galton-Watson processes  is, in a sense, reacher compared to the Markov branching processes.  Only a subclass of embeddable Galton-Watson processes can be obtained from the Markov branching process using  time discretization. For example, the branching process with Poisson distribution for the current population size is only possible in the discrete time setting. 
However, the continuous time setting is easier to analyze, cf \cite{Z}, and it seems to be more logical to start the theory with the direct proofs for  the continuous time branching processes. 

In this paper we give a concise self-contained presentation of key limit theorems for the one-type Markov branching process  $\{Z_t\}_{t\ge0}$ stemming form a single particle alive at time zero. We develop a new approach using a tool which we call tail generating functions. Our proofs are shorter and more transparent than those available in the literature so far. One of the purposes of this paper is to provide a convenient reference for researches using this basic stochastic reproduction model. 

Markov branching processes form a special class of age-dependent branching processes characterized by exponential life lengths. Each particle at the moment of death produces a random number of offspring with probability generating function
$$f(s)=\sum_{k=0}^\infty p_ks^k,$$
where it is always assumed that $p_1<1$.
Denote by $m=f'(1)$ the offspring mean number, and by $\lambda$ the parameter of the exponential distribution for the lifelength.
 In  terms of the population size mean
\begin{equation}\label{Mt}
M_t=E(Z_t)=e^{\lambda(m-1)t},
\end{equation}
three different regimes of reproduction can be discerned: subcritical ($m<1$), critical ($m=1$), and supercritical ($m>1$).

A remarkable feature of Markov branching processes is that the probability generating functions 
\[F_t(s)=E(s^{Z_t}|Z_0=1),\quad t\ge0,\]
satisfy the following integral equation
 \begin{equation}\label{ieq}
\lambda t=\int_s^{F_t(s)}{dx\over f(x)-x}.
\end{equation}
The main challenge in analyzing this equation is to handle the singularity points $x$ satisfying $f(x)=x$. Clearly, one of these singularity points is always $x=1$. Due to convexity of the generating function $f(s)$ for $s\ge0$, we have at most two such non-negative roots. 
\begin{definition}\label{qr}
 Denote by $q\in[0,1]$  the smallest  non-negative root of the equation $f(x)=x$. The second root, if any, will be denoted by $r$, so that $q<r<\infty$.
If $q=1<r$, then the process is called an extendable subcritical  branching process.
\end{definition}

It turns out that $q=P(Z_\infty=0)$ gives the probability of ultimate extinction of the branching process in question. In the subcritical and critical cases we have $q=1$, and the supercritical case is characterized by $0\le q<1=r$.  To make the text self-contained the above mentioned and other basic results  will be quickly established in Section \ref{Smain}. Section \ref{Smain} also presents the main result of the paper introducing refined integral equations for $F_t(s)$ which are obtained from \eqref{ieq} after the principal singularity terms being extracted. 

Sections \ref{Sta}, \ref{Stail}, and \ref{Stai} introduce and develop an instrumental device $\nabla_av(s)={v(s)-v(a)\over s-a}$, called a tail generating function, for working with the generating functions $v(s)$. 
 If $a=1$, and $f(s)=Es^\nu$ is a probability generating function, then the transformation
 \[\nabla_1  f(s)=\sum_{k=0}^\infty s^kP(\nu>k),\]
brings the generating function for the tail probabilities, which is the reason for using the term "tail generating function".

If $\nabla_av(s)$ has the form $c_1+c_2v(s)$, then $v(s)$ must be a linear-fractional function. In particular, for the simplest linear-fractional generating function $v(s)={1\over 1-s}$, we have 
\begin{align*}
 \nabla_{a_1}\ldots\nabla_{a_{n}} v(s)={1\over (1-a_1)\ldots(1-a_{n})(1-s)},
\end{align*}
given $a_1,\ldots,a_n\in[0,1)$.
The illuminating case of the linear-fractional $f(s)$ is discussed in Section \ref{Ex}.

In Section \ref{Sde} we give another angle to the transformations of branching processes connecting a supercritical branching process with $q\in(0,1)$ to a subcritical branching with $q=1$, on one hand, and to a "purely supercritical" branching process with $q=0$, on the other hand.
In Sections \ref{Ssub}, \ref{Scr}, and \ref{Ssup} we apply our approach to the critical, subcritical, and supercritical cases, and give new, streamlined proofs for (updated versions of) the key limit theorems. 

%

\section{Tail generating functions and $x\log x$ condition}\label{Sta}
\begin{definition} \label{def}
 Consider an arbitrary (not necessarily probability) generating function 
 
\begin{equation} \label{vs}
 v(s)=\sum_{k=0}^\infty s^kv_k,\quad v_k\ge0,\quad s\in[0,R],
\end{equation}
with radius of convergence $R\le\infty$. For  a given $a\in[0,R]$,  define a new generating function
$$\nabla_a v(s)={v(s)-v(a)\over s-a},\quad \nabla_a v(a)=v'(a),$$
which we will call  a tail generating function for $v(s)$. For $n\ge1$, define recursively
\[\nabla_a^n v(s)=\nabla_a (\nabla_a^{n-1}  v)(s),\quad\nabla_a^0  v(s)=v(s).\]
\end{definition}

\begin{proposition}\label{a0}
The commutative property  $\nabla_a\nabla_b=\nabla_b\nabla_a$ holds, and for any aligible $(a_1,\ldots,a_{n+1})$, 
\begin{align*}
&\nabla_{a_1}\ldots\nabla_{a_{n}} v(a_{n+1})
=\sum_{k=0}^{\infty} v_{k+n}\sum\limits_{\substack{i_1+\ldots+i_{n+1}=k\\ i_1\ge0,\ldots, i_{n+1}\ge0}}a_1^{i_1}\cdots a_{n+1}^{i_{n+1}}.
\end{align*}
\end{proposition}

\begin{proof} Clearly,
\begin{align*}
\nabla_a  (s^k)&={s^k-a^k\over s-a}=\sum_{i=0}^{k-1}  s^{i}a^{k-1-i},
\end{align*}
and  the stated equality follows for $n=1$:
 \begin{align*}
\nabla_a  v(s)=\sum_{k=1}^\infty v_k\nabla_a(s^k)=\sum_{k=1}^\infty v_k\sum_{i=0}^{k-1}  s^{i}a^{k-1-i}.
\end{align*}
From here, writing 
$$\nabla_bv(s)=\sum_{k=0}^{\infty} s^k u_k,\quad u_k=\sum_{j=0}^{\infty} b^jv_{j+k+1},
$$ 
we find
\begin{equation*}
\nabla_{a}\nabla_{b} v(s)=\sum_{k=0}^\infty u_{k+1}\sum_{i=0}^{k}  s^{i}a^{k-i}
=\sum_{l=0}^\infty v_{l+2}\sum_{i_1+i_2+i_3=l}  s^{i_1}a^{i_2}b^{i_3},
\end{equation*}
giving the statement for $n=2$, which by the symmetry over $a$ and $b$ implies the stated commutativity. The arbitrary $n$ in Proposition \ref{a0} is handled recursively using the same argument.
\end{proof}

\begin{corollary}\label{comb}
We have
\begin{align*}
&\nabla_{a}^nv(s)
=\sum_{k=0}^{\infty} v_{k+n}\sum_{i=0}^ks^i{k-i+n-1\choose n-1}a^{k-i}.
\end{align*}
In particular, if $v^{(n)}(s)$ stands for the $n$-th derivative of $v(s)$, then
\[\nabla_a^n  v(a)={v^{(n)}(a)\over n!},\]
confirming that $v_n=\nabla_0^n  v(0)={v^{(n)}(0)\over n!}$.
\end{corollary}
\begin{proof} The claim follows from  Proposition \ref{a0} and a combinatoric  equality
\begin{align*}
\sum\limits_{\substack{i_1+\ldots+i_{n+1}=k\\ i_1\ge0,\ldots, i_{n+1}\ge0}}a^{i_1}\cdots a^{i_{n+1}}
={k+n\choose n}a^{k}.
\end{align*}
\end{proof}

\begin{proposition}\label{tail}
For given $a\in(0,R]$ and $n\ge1$, the moment condition
\begin{equation}\label{vlv}
 \sum_{k=2}^\infty v_ka^kk^{n-1}\ln k<\infty
\end{equation}
is equivalent to
$$\int_0^a \nabla_a^{n}  v(x)dx<\infty.$$
\end{proposition}
\begin{proof}
By Corollary \ref{comb}, for $n\ge1$,
\begin{align*}
\int_0^a \nabla_a^nv(x)dx&
=\sum_{k=0}^{\infty} v_{k+n}\sum_{i=0}^k{a^{i+1}\over i+1}{k-i+n-1\choose n-1}a^{k-i}
={1\over n!}\sum_{k=0}^{\infty} v_{k+n}a^{k+1}\sum_{i=0}^k{n\over i+1}\prod_{j=1}^{n-1}(k-i+j),
\end{align*}
and it enough to observe that
\[ \sum_{i=0}^k{n\over i+1}\prod_{j=1}^{n-1}(k-i+j)\sim k^{n-1}\ln k,\quad k\to\infty.\]
\end{proof}

\begin{corollary}\label{cor}
For a given generating function \eqref{vs} and an $a\in(0,R]$, the $x\log x$ moment condition 
\begin{equation}\label{xlx}
 \sum_{k=2}^\infty v_ka^kk\ln k<\infty,
\end{equation}
is equivalent to
$$\int_0^a \nabla_a^{2}  v(x)dx<\infty.$$

\end{corollary}

\section{Further properties of the tail generating functions}\label{Stail}

\begin{lemma}\label{abc}
For any $a_i\in[0,R]$,
\begin{align*}
v(s)=v(a_1)+\sum_{i=1}^n(s-a_1)\cdots(s-a_{i-1})\nabla_{a_1}\ldots\nabla_{a_{i-1}}v(a_i)+(s-a_1)\cdots(s-a_n)\nabla_{a_1}\ldots\nabla_{a_n} v(s)&.
\end{align*}
In particular, we have the following form of the Taylor polynomial
\begin{equation*}
 v(s)=\sum_{i=0}^{n-1}\nabla_a^i  v(a)(s-a)^i+\nabla_a^n  v(s)(s-a)^{n}.
\end{equation*}

\end{lemma}

\begin{proof} 
The statement follows from
\begin{equation*}
\nabla_{a_1}\ldots\nabla_{a_n} v(s)={v(s)-v(a_1)\over (s-a_1)\cdots(s-a_n)}-{\nabla_{a_1}v(a_2)\over (s-a_2)\cdots(s-a_n)}-\ldots-{\nabla_{a_1}\ldots\nabla_{a_{n-1}}v(a_n)\over s-a_n},
\end{equation*}
which is obtained recursively from Definition \ref{def}.
\end{proof}
\begin{lemma}\label{ab}
For $a\ne b$, 
\begin{align*}
\nabla_{a}\nabla_{b} v(b)&=\nabla_{b}^2 v(a)={v'(b)-\nabla_bv(a)\over b-a},\qquad 
\nabla_{a}\nabla_{b} v(a)=\nabla_{a}^2 v(b)={v'(a)-\nabla_av(b)\over a-b},\\
\nabla_{a}^2\nabla_{b} v(a)&={v''(a)\over2(a-b)}-{v'(a)-\nabla_av(b)\over(a-b)^2},\qquad
 \nabla_{a}^2\nabla_{b} v(b)={v'(a)+v'(b)-2\nabla_bv(a)\over(a-b)^2}.
\end{align*}
 
\end{lemma}

\begin{proof} These relations are obtained from Definition \ref{def} in a straightforward way. For example,
\begin{align*}
\nabla_b\nabla_{a}^2v(a)={\nabla_a^2v(a)-\nabla_a^2v(b)\over a-b}={v''(a)\over2(a-b)}-{v'(a)-\nabla_av(b)\over(a-b)^2}.
\end{align*}


\end{proof}
Any generating function $v(s)$ is convex over $[0,R]$ and equation $v(x)=x$ has at most two non-negative roots. In the spirit of Definition \ref{qr}, assuming that there exists at least one such root,
 we will denote by $q\in[0,\infty)$ the smallest non-negative root of the equation $v(x)=x$. If the second root $r$ exists, then $r\in(q,\infty)$.  We right $q=r=\infty$ if  $v(x)>x$ for all $x\ge0$. We write $r=\infty$ when there is only a single root $q$ in the interval $[0,R]$.
\begin{corollary}\label{casa}
If $q<\infty$ and $v'(q)=1$, then
\begin{align*}
v(s)-s&=(q-s)^2\nabla_q^2v(s).
\end{align*}
\end{corollary}

\begin{corollary}\label{cas}
If $q<\infty$, then
\begin{align*}
v(s)-s&=(q-s)(1-\nabla_qv(s)).
\end{align*}
If moreover $r<\infty$, then
\begin{align*}
s-v(s)&=(r-s)(\nabla_rv(s)-1).
\end{align*}
\end{corollary}

\begin{corollary}\label{aob}
If $q<r<\infty$, then $\nabla_qv(r)=\nabla_rv(q)=1$, $v'(q)<1<v'(r)$, and
\begin{align*}
\nabla_q v(s)&=1-(r-s)\nabla_q\nabla_r v(s),\qquad \nabla_r v(s)=1+(s-q)\nabla_q\nabla_r v(s),\\
\nabla_{q}\nabla_{r} v(q)&=\nabla_q^2v(r)={1-v'(q)\over r-q},\qquad \nabla_{q}\nabla_{r} v(r)=\nabla_r^2v(q)={v'(r)-1\over r-q}.
\end{align*}
Since $\nabla_{q}\nabla_{r} v(q)<\nabla_{q}\nabla_{r} v(r)$, we conclude
\begin{equation}\label{beta}
 \beta:={\nabla_{q}\nabla_{r} v(q)\over \nabla_{q}\nabla_{r} v(r)}\in(0,1).
\end{equation}
\end{corollary}

\begin{lemma}\label{bet}
If $q<r<\infty$, then
\begin{align*}
{\nabla_q^2 v(s)\over 1-\nabla_qv(s)}&={\beta\over r-s}+\beta{\nabla_r^2 \nabla_qv(s)\over \nabla_r\nabla_qv(s)}
-{\nabla_r \nabla_q^2v(s)\over \nabla_r\nabla_qv(s)},\\
{\nabla_r^2 v(s)\over \nabla_rv(s)-1}&={1\over \beta(s-q)}+\beta{\nabla_r^2 \nabla_qv(s)\over \nabla_r\nabla_qv(s)}
-{\nabla_r \nabla_q^2v(s)\over \nabla_r\nabla_qv(s)}.
\end{align*}

\end{lemma}
\begin{proof} By Corollary \ref{aob},
 \begin{align*}
\nabla_{q}\nabla_{r} v(r){\nabla_q^2 v(s)\over 1-\nabla_qv(s)}
&={\nabla_{q}\nabla_{r} v(r)\nabla_q^2v(s)\over (r-s)\nabla_r\nabla_qv(s)}
={\nabla_q^2v(r)\over r-s}+{\nabla_{r}\nabla_{q} v(r) \nabla_q^2v(s)-\nabla_r\nabla_qv(s)\nabla_q^2v(r)\over (r-s)\nabla_r\nabla_qv(s)}\\
&={\nabla_{q}\nabla_{r} v(q)\over r-s}+{\nabla_{q}^2 v(r)\nabla_r^2 \nabla_qv(s)-\nabla_r\nabla_qv(r)\nabla_r\nabla_q^2v(r)\over \nabla_r\nabla_qv(s)}.
\end{align*}
Dividing both sides by $\nabla_{q}\nabla_{r} v(r)$ we get the first stated equality. The second equality is obtained similarly using
 \begin{align*}
\nabla_{q}\nabla_{r} v(q){\nabla_r^2 v(s)\over \nabla_rv(s)-1}
&={\nabla_{q}\nabla_{r} v(q)\nabla_r^2v(s)\over (s-q)\nabla_r\nabla_qv(s)}
=
{\nabla_{q}\nabla_{r} v(r)\over s-q}+{\nabla_r\nabla_qv(q)\nabla_q\nabla_r^2v(q)-\nabla_{r}^2 v(q)\nabla_r \nabla_q^2v(s)\over \nabla_r\nabla_qv(s)}.
\end{align*}

\end{proof}

\begin{corollary}\label{abb}
If $q<r<\infty$, then
\begin{align*}
\nabla_{q}^2\nabla_{r} v(q)&={1-v'(q)\over(r-q)^2}-{v''(q)\over2(r-q)},\\
 \nabla_{q}^2\nabla_{r} v(r)&={v'(q)+v'(r)-2\over(r-q)^2}.
\end{align*}

\end{corollary}

\section{A family of $\pi$ functions}\label{Stai}
Given $q<\infty$, define $\pi(s_1,s_2)$ via
 \begin{align}
\pi(s_1,s_2)&=\int_{s_1}^{s_2}{dx\over v(x)-x},\quad 0\le s_1\le s_2<q,\label{pid}\\
\pi(s_1,s_2)&=\int_{s_1}^{s_2}{dx\over v(x)-x},\quad q< s_2\le s_1<r\wedge R.\label{pidu}
\end{align}
In view of \eqref{ieq}, studying the properties of such a function with $v(s)=f(s)$ is very important for the analysis of the Markov branching processes. 
\begin{lemma}\label{picr} If  $q<\infty$ and  $v'(q)=1$, then 
\[
\pi(s):=\pi(0,s)= \int_{0}^{s}{dx\over (q-x)^2\nabla_q^{2}  v(x)},\quad 0\le s<q.
 \]
\end{lemma}
\begin{proof}
The claim immediately follows from Corollary \ref{casa}.
\end{proof}

\begin{lemma}\label{pisub} If  $q<\infty$ and  $v'(q)<1$, then 
  \[
(1-v'(q))\pi(s_1,s_2)=\ln{q-s_1\over q-s_2}  +\pi_q(s_1)-\pi_q(s_2),
\]
where
 \begin{align}\label{piq}
\pi_q(s)=\int_0^{s}{\nabla_q^{2}  v(x)dx\over 1-\nabla_q  v(x)},\quad 0\le s<r\wedge R.
\end{align}
\end{lemma}
\begin{proof}
By Corollary \ref{cas},
\[(1-v'(q))\int_{s_1}^{s_2}{dx\over v(x)-x}=\int_{s_1}^{s_2}{1-\nabla_q v(q)\over (q-x)(1-\nabla_q v(x))}dx=\int_{s_1}^{s_2}{dx\over q-x}-\int_{s_1}^{s_2}{\nabla_q^2 v(x)dx\over 1-\nabla_q v(x)},\]
implying the assertion. Notice that $\pi_q(s)$ is a generating function due to
\[\pi_q(s)=\sum_{k=0}^\infty\int_0^{s}\nabla_q^2 v(x) (\nabla_q v(x))^kdx.\]
\end{proof}

\begin{lemma}\label{pisu} If  $r<\infty$, then for $q< s_2\le s_1<r$,
  \[
(v'(r)-1)\pi(s_1,s_2)=\ln{r-s_1\over r-s_2}  +\pi_r(s_2,s_1),
\]
where
 \begin{align}\label{piqu}
\pi_r(s_1,s_2)=\int_{s_1}^{s_2}{\nabla_r^{2}  v(x)dx\over \nabla_r  v(x)-1},\quad q< s_1\le s_2<r.
\end{align}
\end{lemma}
\begin{proof}
By Corollary \ref{cas}, for $q<s_1<s_2<r$,
\[(v'(r)-1)\int_{s_1}^{s_2}{dx\over x-v(x)}=\int_{s_1}^{s_2}{\nabla_r v(r)-1\over (r-x)(\nabla_r v(x)-1)}dx=\int_{s_1}^{s_2}{dx\over r-x}-\int_{s_1}^{s_2}{\nabla_r^2 v(x)dx\over \nabla_r v(x)-1},\]
implying the assertion. 
\end{proof}

\begin{lemma}\label{pisup} If  $r<\infty$, then 
\begin{align*}
 \pi_q(s)&
=\beta\ln{1\over r-s}+\pi_{rq}(s)-\pi_{qr}(s),\quad 0\le s<r,\\
\pi_r(s_1,s_2)&=\beta^{-1}\ln{1\over s-q}+\pi_{rq}(s_2)-\pi_{rq}(s_1)-\pi_{qr}(s_1)+\pi_{qr}(s_2),\quad q< s_1<s_2<r,
\end{align*}
\[\]
where $\beta$ is given by \eqref{beta} and 
 \begin{align}\label{1q}
\pi_{rq}(s)&=\beta\int_0^s{\nabla_r^2\nabla_q  v(x)\over \nabla_r\nabla_q  v(x)}dx,\quad 
\pi_{qr}(s)=\int_0^s{\nabla_r\nabla_q^2  v(x)\over \nabla_r\nabla_q  v(x)}dx.
\end{align}
\end{lemma}
\begin{proof}
Use Lemma \ref{bet}.\end{proof}

\begin{proposition}\label{pipi}
Consider the $\pi$ functions defined by \eqref{piq} and \eqref{1q}.

(i) If  $q<\infty$, then $\mathcal L_q(x)=e^{\pi_q(q-x)}$ slowly varies at zero, and $\pi_q(q)<\infty$ if and only if the  $x\log x$ condition \eqref{xlx} holds with $a=q$. If $r<\infty$, then $\pi_{q}(q)<\infty$.

(ii) If  $r<\infty$, then $\mathcal L_{rq}(x)=e^{\pi_{rq}(r-x)}$ slowly varies at zero, and $\pi_{rq}(r)<\infty$ if and only if the  $x\log x$ condition \eqref{xlx} holds with $a=r$. 

(iii)  If $r<\infty$, then $\pi_{qr}(r)<\infty$.
\end{proposition}
\begin{proof}
 Use Corollary \ref{cor} to see that   $\pi_q(q)<\infty$ is equivalent to \eqref{xlx} with $a=q$. Slow variation of $\mathcal L_q(x)$ is seen via the representation
 \[\mathcal L_q(x)=e^{\int_x^q{\epsilon(s)ds\over s}},\quad \epsilon(q-s)={(q-s)\nabla_q^{2}  v(s)\over 1-\nabla_q  v(s)}
 ={\nabla_q  v(s)-v'(q)\over 1-\nabla_q  v(s)},\]
 where $\epsilon(x)\to0$ as $x\to0$.
 If $r<\infty$, then 
 $$\pi_{q}(q)<{\nabla_q^2  v(r)\over 1-\nabla_q  v(0)}={1-v'(q)\over(r-q)(1-\nabla_q  v(0))}<\infty,$$
 finishing the proof of $(i)$.
Turning to part $(ii)$, observe that
since 
\[\int_{r-q/2}^r\nabla_q\nabla_r^2 v(x)dx=\int_{r-q/2}^r{\nabla_r^2 v(x)-\nabla_r^2 v(q)\over x-q}dx= \int_{r-q/2}^r{\nabla_r^2 v(x)\over x-q}dx-{v'(r)-1\over r-q}\ln{2(r-q)\over q},\]
the following two inequalities are equivalent
\[\int_{r-q/2}^r\nabla_q\nabla_r^2 v(x)dx<\infty,\quad \int_{0}^r\nabla_r^2 v(x)dx<\infty. \]
Thus indeed,  by Corollary \ref{cor}, $\pi_{rq}(r)<\infty$ is equivalent to \eqref{xlx}  with $a=r$.
Slow variation of $\mathcal L_{rq}(x)$ follows from the representation
 \[\mathcal L_{rq}(x)=e^{\int_x^r\epsilon_{rq}(s){\beta ds\over s}},\quad \epsilon_{rq}(r-s)={(r-s)\nabla_r^{2}\nabla_q  v(s)\over \nabla_r\nabla_q  v(s)}
 ={\nabla_r\nabla_q  v(r)-\nabla_r\nabla_q  v(s)\over \nabla_r\nabla_q  v(s)},\]
 where $\epsilon_{rq}(x)\to0$ as $x\to0$. Hence $(ii)$ holds.
Finally, $(iii)$ follows from
 $$\pi_{qr}(r)<{\nabla_r\nabla_q^2  v(r)\over \nabla_r\nabla_q  v(0)}<\infty,$$
where by Corollary \ref{abb}
\[
\nabla_r\nabla_q  v(0)=\left\{
\begin{array}{lll}
  1-v_1,& \mbox{for }  v_0=0,&   \\
 v_0/q &  \mbox{for }  v_0>0  &   ,
 \end{array}
\right.\quad
\nabla_r\nabla_q^2  v(r)={v'(r)+v'(q)-2\over(r-q)^2}.
\]

\end{proof}

\section{Probability generating functions of the branching process}\label{Smain}

We turn to the probability generating functions $F_t(s)=Es^{Z_t}$ and start by deriving the integral equation \eqref{ieq}. Afterwards, we prove the main finding of this paper, Theorem \ref{main}, presenting refinements of the  equation \eqref{ieq} in terms of the tail generating functions. For $F(s)=F_t(s)$ we will use notation $\nabla_aF_t(s)=\nabla_aF(s)$ and $F_t'(s)=F'(s)$.

If $T$ and  $\nu$ are the life length and offspring number of the ancestral particle, then the following branching renewal property 
 \[Z_t=1_{\{T>t\}}+1_{\{T\le t\}}\sum_{i=1}^{\nu}Z_{t-T}^{(i)}\]
holds, with  $Z_{t-T}^{(i)}$ standing for the number of descendants from the $i$-th ancestral daughter. By the assumption of exponential life length and independence among daughter particles, the branching property yields
 \[F_t(s)=se^{-\lambda t}+\lambda\int_0^tf(F_{t-u}(s))e^{-\lambda u}du,\]
 or more conveniently,
  \begin{equation*}
F_t(s)e^{\lambda t}=s+\lambda\int_0^tf(F_u(s))e^{\lambda u}du.
\end{equation*}
Taking the derivatives we arrive at the backward Kolmogorov equation for the Markov process $\{Z_t\}$
 \begin{equation}\label{ode}
{\partial F_t(s)\over \partial t}=\lambda\Big[f(F_t(s))-F_t(s)\Big],\qquad F_0(s)=s,
\end{equation}
leading to \eqref{ieq}. 
For $M_t=F'_t(1)$, the ordinary differential equation  \eqref{ode} yields  $M_t'=\lambda(m-1)M_t$ with $M_0=1$. This brings the exponential growth formula \eqref{Mt}. 

\begin{proposition}\label{qF} If $q$ is the smallest non-negative root of $f(x)=x$, then  $P(Z_\infty=0)=q$ and  $F_t(q)=q$ for all $t\ge0$. Moreover, $F_t(s)\to q$ as $t\to\infty$ for $s\in[0,1)$.
\end{proposition}
\begin{proof}
Let $q_\infty$ stand for the extinction probability $P(Z_\infty=0)$ which is the limit of the monotone function
 \[P(Z_t=0)=F_t(0)\nearrow q_\infty,\quad t\to\infty.\]
We want to show that $q_\infty=q$. From ${\partial F_t(0)\over \partial t}>0$ we see that  $F_t(0)<q$, since $f(F_t(0))>F_t(0)$ in accordance with \eqref{ode}. Thus $q_\infty \le q$.
Moreover, since
$$q_\infty =E(E(Z_\infty=0|Z_t))=E(q_\infty ^{Z_t})=F_t(q_\infty ),\quad t\ge0,$$
equation \eqref{ieq} entails $q_\infty =f(q_\infty )$. 
\end{proof}
\begin{corollary}
Equation \eqref{ieq} can be rewritten as 
 \begin{align}\label{pi}
\pi(s,F_t(s))=\lambda t
\end{align}
in terms of $\pi(s_1,s_2)$ defined by  \eqref{pid}-\eqref{pidu} for $v(s)=f(s)$.
\end{corollary}

\begin{proposition}\label{reg} A supercritical Markov branching process $\{Z_t\}$ with the reproduction law $f(s)$ is regular, that is $P(Z_t<\infty)=1$ for all  $t>0$,  if and only if
  \begin{align}\label{regu}
\int_{1-\epsilon}^1{dx\over x-f(x)}=\infty.
\end{align}
\end{proposition}
\begin{proof} By \eqref{ieq}, we have for all $t\ge0$, 
\[\int_{F_t(s_1)}^{F_t(s_2)}{dx\over x-f(x)}=\int_{s_1}^{s_2}{dx\over x-f(x)},\quad q<s_1<s_2<1.\]
Letting $s_1=s$ and $s_2\nearrow1$ we get
\[\int_{F_t(s)}^{F_t(1)}{dx\over x-f(x)}=\int_{s}^{1}{dx\over x-f(x)},\quad q<s<1.\]
This reveals an important dichotomy: either
\[\int_{1-\epsilon}^1{dx\over x-f(x)}<\infty,\]
and $F_t(1)\in(0,1)$ satisfies
  \[\int_{F_t(1)}^1{dx\over x-f(x)}=\lambda t,\quad t\ge0,\]
  or \eqref{regu} holds and the branching process is regular, that is $F_t(1)=1$ for all $t\ge0$.
  
 In particular, the Markov branching process is regular provided $m<\infty$. Indeed, by Corollary \ref{cas},
 \[
f(s)-s=(1-s)(1-\nabla_1  f(s)),\quad \nabla_1  f(1)=m,
\]
implying the regularity condition \eqref{regu}.
\end{proof}

Next comes the main result of the paper. We will use notation from the previous section adjusted to the probability generating function $v(s)=f(s)$. In this case $q\le1\le R$, and if $q<1$, then $r=1$.
\begin{theorem}\label{main}
If  $t\ge0$ and  $s\in[0,1)$, then

(i) for $m=1$,
\begin{equation}\label{creq}
 \int_{s}^{F_t(s)}{dx\over (1-x)^2\nabla_1^{2}  f(x)}=\lambda t,
 \end{equation}

(ii) for $m\ne1$, we have $F_t'(q)=\gamma^{t}$, where $\gamma=e^{\lambda(f'(q)-1) }\in(0,1)$, and
 \begin{align}\label{subeq}
\nabla_qF_t(s)=\gamma^{t} \exp\Big\{-\int_s^{F_t(s)}{\nabla_q^{2}  f(x)dx\over 1-\nabla_q  f(x)}\Big\},
\end{align}

(iii) for $1<m<\infty$,  we have   $\beta={1-f'(q)\over m-1}\in(0,1)$ and 
 \begin{align}\label{supeq}
\nabla_qF_t(s)=\gamma^{t} \big[\nabla_1F_t(s)\big]^{\beta}
\exp\Big\{\int_s^{F_t(s)}{\nabla_1\nabla_q^{2}  f(x)-\beta\nabla_1^2\nabla_q  f(x)\over \nabla_1\nabla_q  f(x)}dx\Big\}.
\end{align}

\end{theorem}
\begin{proof} 
Claim $(i)$ follows from  \eqref{pi} and Lemma \ref{picr}.
For $m\ne1$, combining Lemma \ref{pisub}
and  \eqref{ieq}, brings
 \[
(1-f'(q))\lambda t=\ln{q-s\over q-F_t(s)}  +\pi_q(s)-\pi_q(F_t(s)).
\]
Thus
 \begin{align}\label{supic}
\nabla_q  F_t(s)e^{\pi_q(F_t(s))}=\gamma^{t} e^{\pi_q(s)},
\end{align}
$F_t'(q)=\nabla_q  F_t(q)=\gamma^{t}$, and claim $(ii)$ follows. Similarly, claim $(iii)$ follows from  Lemma \ref{pisup}.

\end{proof}

\section{Decomposition of the branching process with $0<q<1$}\label{Sde}
This section is devoted to a supercritical branching process $\{Z_t\}$ with $0<q<1$. Depending on the two possible fates of the process, survival $Z_\infty>0$ or extinction  $Z_\infty=0$, we will label the ancestral particle either as successful (with probability $1-q$) or unsuccessful (with probability $q$). 
Similarly, each daughter (if any)  of the ancestral particle will have one of two possible fates: the branching process stemming from this daughter either dies our survives forever. Thus we can view the offspring number $\nu=\nu_1+\nu_2$ as the sum of two components, where $\nu_1$ stands for the number of successful daughters and $\nu_2$ stands for the number of unsuccessful daughters. Due to the independence of the evolutions of new particles we have
\begin{align*}
E(s_1^{\nu_1}s_2^{\nu_2}|Z_\infty>0)&={E(s_1^{\nu_1}s_2^{\nu_2};Z_\infty>0)\over P(Z_\infty>0)}={E(s_1^{\nu_1}s_2^{\nu_2})-P(Z_\infty=0)\over 1-q}\\
&={E\prod_{k=1}^\nu (s_11_{\{Z_\infty^{(k)}>0\}}+s_21_{\{Z_\infty^{(k)}=0\}})-q\over 1-q}={f(s_1(1-q)+s_2q)-q\over 1-q},
\end{align*}
so that 
\begin{align*}
E(x^{\nu_1}|Z_\infty>0)&={f(x(1-q)+q)-q\over 1-q}.
\end{align*}
On the other hand,
\[E(s^{\nu_2}|Z_\infty=0)=E(s^\nu|Z_\infty=0)=
 {E(s^\nu;Z_\infty=0)\over q}={E(\prod_{k=1}^\nu s1_{\{Z_\infty^{(k)}=0\}})\over q}={f(sq)\over q}.\]
 As a result we get a picture of the subcritical one-type branching processes as a two-type branching process where type 1 particles give birth to at least one particle of the same type and a random number of type 2 particles, while the type 2 particles produce only particles of the same type in the subcritical regime. All particles, irrespective of the type, have the same exponential distribution of the life length.

We show next, using this decomposition, that in the intermediate case of $q\in(0,1)$ the key equation \eqref{subeq} split over two domains $s\in[0,q]$ and $s\in [q,1]$, can be recovered with help of simple transformations from the equation
\eqref{subeq} with $q=1$ and $q=0$ respectively.

Consider the branching process $\{X_t\}$ formed by the unsuccessful particles having the dual reproduction law $g(s)={f(sq)\over q}$. Clearly, the new branching process is subcritical with the offspring mean $h'(q)\in(0,1)$.  Notice that with $x=s/q$, $s\in[0,q]$, we have
\begin{align*}
\nabla_1g(x)&=\nabla_q f(s),\\
\nabla_1^2g(x)&={\nabla_q f(q)-\nabla_q f(s)\over 1-x}=q\nabla_q^2 f(s),
\end{align*}
Therefore,
\begin{align*}
{\nabla_1^{2}  g(x)\over 1-\nabla_1  g(x)}={q\nabla_q^{2}  f(y)\over 1-\nabla_q  f(y)},\quad x=y/q\in[0,1],
\end{align*}
and applying \eqref{subeq}  to $G_t(s)=Es^{X_t}$  we find
\begin{align*}
\nabla_1G_t(s)= \gamma^{t}\exp\Big\{-\int_s^{G_t(s)}{\nabla_1^{2}  g(x)dx\over 1-\nabla_1  g(x)}\Big\}= \gamma^{t}\exp\Big\{-\int_{qs}^{qG_t(s)}{\nabla_q^{2}  f(y)dy\over 1-\nabla_q  f(y)}\Big\}.
\end{align*}
Comparing this with \eqref{subeq} for $F_t(s)$, we see that 
\begin{equation}\label{GF} 
\nabla_1G_t(s/q)=\nabla_qF_t(s),\qquad G_t(s)={F_t(sq)\over q}, \qquad t\ge0.
\end{equation}
A proper interpretation of \eqref{GF} is that the subcritical branching process $X_t$ is the supercritical branching process $Z_t$ conditioned on extinction:
\[Es^{X_t}={F_t(sq)\over q}={E(\prod_{k=1}^{Z_t} s1_{\{Z_\infty^{(k)}=0\}})\over q}=
 {E(s^{Z_t};Z_\infty=0)\over q}=E(s^{Z_t}|Z_\infty=0),\]
 see \cite{JL} for a more general statement of this kind. In other words, we demonstrated that the $0\le s\le q$ part of \eqref{subeq} with $q\in(0,1)$ is obtained from \eqref{subeq} with $q=1$  by the transformation 
\begin{equation}\label{co1} 
F_t(s)=qG_t(s/q),\quad s\in[0,q].
\end{equation}

Another useful transformation is based on the branching process $Y_t$ formed by the successful particles having the reproduction law
\[h(s)={f(s(1-q)+q)-q\over 1-q}.\]
Observe that the generating function $h(s)$ is well-defined for $s\in[-{q\over 1-q},1]$. 
With $x={s-q\over 1-q}$ and  $s\in[q,1]$, we have
\begin{align*}
\nabla_0h(x)&=\nabla_qf(s),\\
\nabla_0^2h(x)&={\nabla_qf(s)-\nabla_qf(q)\over x}=(1-q)\nabla_q^2f(s),
\end{align*}
so that 
\begin{align*}
{\nabla_0^{2}  h(x)\over 1-\nabla_0  h(x)}={(1-q)\nabla_q^{2}  f(y)\over 1-\nabla_q  f(y)},\quad x={y-q\over1-q}\in[0,1].
\end{align*}
Since $h(0)=0$ and $h'(0)=\gamma$,  after applying \eqref{subeq}  to $H_t(s)=Es^{Y_t}$ we get
\begin{align*}
\nabla_0H_t(s)= \gamma^{t}\exp\Big\{\int_{H_t(s)}^s{\nabla_0^{2}  h(x)dx\over 1-\nabla_0  h(x)}\Big\}= \gamma^{t}\exp\Big\{\int_{H_t({s(1-q)+q})}^{s(1-q)+q}{\nabla_q^{2}  f(y)\over 1-\nabla_q  f(y)dy}\Big\}.
\end{align*}
Comparing this with \eqref{subeq} for $F_t(s)$, we see that
$$\nabla_0H_t(s)=\nabla_qF_t(s(1-q)+q),\qquad H_t(s)={F_t(s(1-q)+q)-q\over 1-q},$$ 
which gives
\begin{equation}\label{co2} 
F_t(s)=q+(1-q) H_t\big({s-q\over 1-q}\big),\quad s\in[q,1].
\end{equation}

One of the conclusions of this section is that in some questions concerning non-critical Markov branching processes it is enough to investigate in detail a subcritical generating function, $G_t(s)$, and a "purely supercritical" generating function $H_t(s)$ with $H_t(0)=0$. Then the intermediate supercritical case can be addressed using the transformations \eqref{co1} and \eqref{co2}.

\begin{lemma}\label{A}
For a subcritical extendable $f(s)$, when there exists $r>1$, such that $f(r)=r$, we have  $F_t(r)=r$ for all $t\ge0$.
\end{lemma}
\begin{proof}
 Consider a branching process $X_t$ with the reproduction law $g(s)=f(rs)/r$. This is a  supercritical regular process with $G_t(s)=Es^{X_t}$ satisfying $G_t(1)=1$  for all $t\ge0$. The statement follows from the equality $G_t(s)=F_t(rs)/r$, which is established in the same way as \eqref{GF}.
\end{proof}
\section{Tail generating functions of the linear-fractional form}\label{Ex}

We illustrate our technique using the linear-fractional reproduction law
\begin{equation}\label{lf} 
f(s)=p_0+(1-p_0){ps\over1-(1-p)s},\quad p\in[0,1],\quad p\in(0,1].
\end{equation}
Notice that in contrast to the discrete time case, see for example  \cite{S}, here the linear-fractional reproduction law does not imply the linear-fractional distribution for $Z_t$. It is easy to check that for any $n\ge1$ and $k\ge0$, the tail generating functions are also linear-fractional
\begin{equation}\label{kn} 
\nabla_{a_1}\ldots\nabla_{a_n}f(s)={p(1-p_0)\over1-p }\prod_{i=1}^n{1-p\over1-(1-p)a_i}\cdot{1\over1-(1-p)s}.
\end{equation}
In particular, 
\begin{align*}
\nabla_1^nf(s)&={(1-p)^{n-1}\over p^{n-1}}\cdot{1-p_0\over1-(1-p)s},\\
p_n&=\nabla_0^nf(0)=(1-p_0)(1-p)^{n-1}p,
\end{align*}
The last equality implies that conditioned on being positive,  the offspring number distribution is shifted geometric with parameter $p$.

Consider separately the three major regimes of reproduction depending of the mean offspring number  $m={1-p_0\over p}$. In the critical case, $p_0=1-p$, we have
$$\int_0^{s}{dx\over (1-x)^2\nabla_1^2f(x)}=\int_0^{s}{(1-(1-p)x)dx\over (1-p)(1-x)^2}
=\ln{1\over1-s}+{ps\over (1-p)(1-s)},$$ 
so that  equation \eqref{creq} takes the form
$$\ln \nabla_1F_t(s)={pF_t(s)\over (1-p)(1-F_t(s))}-{ps\over (1-p)(1-s)}-\lambda t,$$ 
so that by Corollary \ref{casa},
$$\nabla_1F_t(s)={p\over 1-p}\cdot{\nabla_2F_t(s)\over \lambda t+\ln \nabla_1F_t(s)}.$$ 

Turning to the subcritical case, $p+p_0>1$, observe first that we get an extendable subcritical process with
\begin{align*}
f(r) &=r,\quad r={p_0\over1-p}>1,
\end{align*}
and according to Lemma \ref{A}, 
\[\nabla_rF_t(s)={r-F_t(s)\over r-s}.\]
This, together with
\begin{align*}
\pi_1(s) &=\int_0^{s}{\nabla_1^{2}  f(x)dx\over 1-\nabla_1  f(x)}=\int_0^s{(1-p_0)(1-p)dx\over p(p_0-(1-p)x)}
=m\ln {r\over r-s},  
\end{align*}
leads to  the following compact form for \eqref{subeq} with $m<1$
\begin{align}\label{lfsub}
\nabla_1F_t(s)=e^{-\lambda (1-m)t}\big[\nabla_rF_t(s)\big]^{m}.
\end{align}

In the supercritical case, $p+p_0<1$, the extinction probability is $q={p_0\over1-p}<1$. By \eqref{kn},
\begin{align*}
\nabla_1\nabla_q  f(s)&={1-p\over1-(1-p)s},\\
\nabla_1^2\nabla_q  f(s)&={(1-p)^2\over p(1-(1-p)s)},\\
\nabla_1\nabla_q^2  f(s)&={(1-p)^2\over (1-p_0)(1-(1-p)s)}.
\end{align*}
Taking into account
 $$\beta={\nabla_1\nabla_q  f(q)\over \nabla_1\nabla_q  f(1)}={p\over1-p_0}=1/m,$$ 
we get
\[\int_0^s{\nabla_1\nabla_q^{2}  f(x)-\beta\nabla_1^2\nabla_q  f(x)\over \nabla_1\nabla_q  f(x)}dx=0.\]
Thus, equation  \eqref{supeq} in the linear-fractional case becomes very simple: we have $f'(q)=1/m$ and
\[\nabla_qF_t(s)=e^{-\lambda (1-1/m)t} \big[\nabla_1F_t(s)\big]^{1/m}.\]
Notice the obvious duality between this equation and its counterpart \eqref{lfsub} for the subcritical case.


Finally, in the linear-fractional case the decomposition of the supercritical branching process is valid with
\begin{align*}
E(s_1^{\nu_1}s_2^{\nu_2}|Z_\infty>0)&={ps_1\over 1-(1-p-p_0)s_1-p_0s_2}\cdot{1-p_0\over 1-p_0s_2}.
\end{align*}

\section{Conditional limit distribution in the subcritical case}\label{Ssub}
Denote by $Q_t=1-F_t(0)=P(Z_t>0)$ the probability of survival by time $t$. 
\begin{proposition}\label{bcr}
If $m<1$, then as $t\to\infty$
\begin{align*}
e^{t\lambda(1-m)}Q_t&\to c\in[0,\infty),\\
E(s^{Z_t}|Z_t>0)&\to \psi(s),\quad s\in[0,1],
\end{align*}
where the limit probability generating function $\psi(s)$ is determined by
 \begin{align*}
\nabla_1\psi(s)=\exp\Big\{\int_0^s{\nabla_1^{2}  f(x)dx\over 1-\nabla_1  f(x)}\Big\}.
\end{align*}
Each of the following two cases, $c>0$ and $\psi'(1)<\infty$, is equivalent to the $x\log x$ condition 
\begin{equation}\label{plx}
 \sum_{k=2}^\infty p_kk\ln k<\infty.
\end{equation}
When $c=0$, there is a slowly varying monotone function $\mathcal L_1$ such that $\mathcal L_1(x)\to0$ as $x\to0$, and
\begin{align*}
Q_t\mathcal L_1(Q_t)= e^{-t\lambda(1-m)}.
\end{align*}
\end{proposition}

\begin{proof} According to Theorem \ref{main}  $(ii)$ we have
 \begin{align*}
{1-F_t(s)\over 1-s}=M_t \exp\Big\{-\int_s^{F_t(s)}{\nabla_1^{2}  f(x)dx\over 1-\nabla_1  f(x)}\Big\}.
\end{align*}
Putting here $s=0$ we get
\begin{align*}
Q_t=M_t \exp\Big\{-\int_0^{F_t(0)}{\nabla_1^{2}  f(x)dx\over 1-\nabla_1  f(x)}\Big\},
\end{align*}
and applying Proposition \ref{pipi} $(i)$ we arrive at the stated asymptotic formulae for $Q_t$ with
 $c=e^{-\pi_1(1)}$.
 To establish the stated conditional weak convergence we use
\begin{align*}
E(s^{Z_t}|Z_t>0)={E(s^{Z_t})-P(Z_t=0)\over P(Z_t>0)}=1-{1-F_t(s)\over 1-F_t(0)},
\end{align*}
and the equality
\begin{align*}
{1-F_t(s)\over (1-s)Q_t}=\exp\Big\{\int_0^s{\nabla_1^{2}  f(x)dx\over 1-\nabla_1  f(x)}\Big\}{\mathcal L_1(Q_t)\over \mathcal L_1(1-F_t(s))}.
\end{align*}
We have to verify that ${\mathcal L_1(Q_t)\over \mathcal L_1(1-F_t(s))}\to1$.  But this is true
due to slow variation property of  $\mathcal L_1$ and inequalities
\begin{align*}
1\ge {1-F_t(s)\over Q_t}=1-E(s^{Z_t}|Z_t>0)\ge1-s.
\end{align*}
\end{proof}
\begin{corollary} Consider  the linear-fractional reproduction law \eqref{lf} in the subcritical regime. Proposition \ref{bcr} takes place with $c=({r-1\over r})^{m}$ and 
 \[
\psi(s)=1-(1-s)(1-s/r)^{-m}.
\]
\end{corollary}

\section{Limit theorems in the critical case without higher moments}
\label{Scr}

By \eqref{creq}, the key relation in the critical case is
\[F_t(s)=\pi_{-1}(\pi(s)+\lambda t).\]
Under the classical moment condition allowing for the infinite variance
\begin{equation}\label{alpha}
 f(s)=s+(1-s)^{1+\alpha}\mathcal L(1-x),\quad \alpha\in[0,1],
\end{equation}
where $\mathcal L$ is slowly varying at zero, we can use the properties of regularly varying functions to derive asymptotic results for the critical Markov branching processes. 

If \eqref{alpha} holds with $\alpha>0$, then
\begin{align*}
 &\pi(s)\sim \alpha^{-1}(1-s)^{-\alpha}\mathcal L^{-1}(1-s),\\
&\pi_{-1}(y)= 1-y^{-1/\alpha}\mathcal L^*(y),
\end{align*}
where $\mathcal L^*$ is slowly varying at infinity.
In this case 
$$Q_t=1-\pi_{-1}(\lambda t)\sim(\lambda t)^{-1/\alpha}\mathcal L^*(t).$$
In particular, given the offspring number variance $f''(1)=2b$ is finite, we get $Q_t\sim{1\over b\lambda t}$.
Furthermore,
\[E(e^{-\theta Q_tZ_t}|Z_t>0)=1-{1-F_t(e^{-\theta Q_t})\over Q_t}\to1-(1+\theta^{-\alpha})^{-1/\alpha}, \]
so that in the finite variance case the conditional limit distribution is exponential.

The case $\alpha=0$ is addressed by the next theorem inspired by  its discrete time counterpart from \cite{NW}. 
\begin{theorem}\label{A0}
 If $m=1$ and \eqref{alpha} holds with $\alpha=0$,
then for $x\ge0$,
\[P(V(Z_t)\mathcal  L(Q_t)\le x|Z_t>0)\to 1-e^{-x},\quad t\to\infty,\]
where $V(y)=\pi(1-1/y)$.
\end{theorem}

\begin{proof}
Under the theorem assumptions, $\mathcal L(1-s)=1-\nabla_1f(s)$ is a monotone slowly varying function such that $\mathcal  L(x)\to0$, as $x\to0$.
Therefore,
\[V(y)=\pi(1-1/y)=\int_{1/y}^1{dx\over x\mathcal L(x)}=\int_1^{y}{dz\over z\mathcal L(z^{-1})}\]
implies that $V$ is a monotone slowly varying function such that $V(y)\to\infty$ as $y\to\infty$. By  Theorem 2.4.7 in \cite{Bi}, the inverse of $V$ is rapidly varying so that
\begin{align*}
 V_{-1}(x)/V_{-1}(cx)\to0,\quad x\to\infty,\quad \mbox{ for any }c>1.
 \end{align*}
Thus, for $A_t(x)=V_{-1}(x/L(Q_t))$ and any fixed $0<x<y$, we have $A_t(x)/A_t(y)\to0$.
Therefore, in view of the following inequalities (cf Lemma 1 in \cite{NW})
\[e^{-A_t(x-\epsilon)/A_t(x)}1_{\{Z_t\le A_t(x-\epsilon)\}}\le e^{-Z_t/A_t(x)}\le 1_{\{Z_t\le A_t(x+\epsilon)\}}+e^{-A_t(x+\epsilon)/A_t(x)},\]
it is enough to prove that
\[E(e^{-Z_t/A_t(x)}|Z_t>0)=1-{1-F_t(e^{- 1/A_t(x)})\over Q_t}\to 1-e^{-x},\quad x>0,\quad t\to\infty,\]
or putting $s_t=e^{- 1/A_t(x)}$, that
\begin{equation}\label{ex}
 V_{-1}(\pi(s_t)+\lambda t)\sim e^{x}/Q_t,\quad t\to\infty.
\end{equation}

Using monotonicity of the involved functions, we obtain
\[\pi(s_t)< \pi(1-1/A_t(x))=V(A_t(x))=x/\mathcal L(Q_t),\]
and even
\[(x-\epsilon)/\mathcal L(Q_t)<\pi(s_t)< x/\mathcal L(Q_t),\]
for sufficiently large $t$. On the other hand, from
\[V(e^x/Q_t)-V(1/Q_t)=\int_{1/Q_t}^{e^x/Q_t}{dz\over z\mathcal L(z^{-1})}\sim x/\mathcal L(Q_t)\]
it follows that
\[x/\mathcal L(Q_t)+\lambda t\le V(e^x/Q_t)\le x/\mathcal L(e^{-x}Q_t)+\lambda t.\]
We see that
\[ V_{-1}(\pi(s_t)+\lambda t)\le V_{-1}(x/\mathcal L(Q_t)+\lambda t)\le e^{x}/Q_t,\]
and for sufficiently large $t$, 
\[ V_{-1}(\pi(s_t)+\lambda t)\ge V_{-1}((x-\epsilon)/\mathcal L(Q_t)+\lambda t)\ge V_{-1}((x-2\epsilon)/\mathcal L(e^{-x}Q_t)+\lambda t)\ge e^{x-2\epsilon}/Q_t.\]
Thus \eqref{ex} holds and Theorem \ref{A0} is proven.
\end{proof}

\section{Two limit theorems in the supercritical case}\label{Ssup}

For a supercritical case with $1<m<\infty$ we prove two asymptotic results, Theorems \ref{lsup} and  \ref{gsup}. 

\begin{theorem}\label{lsup}
 Consider a Markov branching process with $1<m<\infty$. Then for $k\ge1$,
 \begin{align*}
P(Z_t=k)&\sim a_k\gamma^{t},\quad t\to\infty,
\end{align*}
 where, see \eqref{piq}, 
 \[\sum_{k=1}^\infty a_ks^k=qe^{-\pi_q(q)}+(s-q) e^{\pi_q(s)-\pi_q(q)}.\]
 Moreover, if $q\in(0,1)$, then
 \[E(s^{Z_t}|Z_t>0,Z_\infty=0)\to 1-(1-s) e^{\pi_q(s)}.\]
\end{theorem}
\begin{proof}  If $p_0=0$, then $F_t(0)=0$. If $p_0>0$, then by  \eqref{supic},
  \begin{align*}
(q-F_t(0))&\sim \gamma^{t}qe^{-\pi_q(q)},\quad t\to\infty,
\end{align*}
and recalling Lemma \ref{qF} we get
  \begin{align*}
F_t(s)-q&\sim \gamma^{t}(s-q)e^{\pi_q(s)-\pi_q(q)},\quad t\to\infty.
\end{align*}
Now, for the first claim, it remains to notice that
\begin{align*}
\sum_{k=1}^\infty P(Z_t=k)s^k=F_t(s)-q+q-F_t(0).
\end{align*}
The second claim follows from
   \begin{align*}
E(s^{Z_t}|Z_t>0,Z_\infty=0)&={E(s^{Z_t};Z_t>0,Z_\infty=0)\over P(Z_t>0,Z_\infty=0)}
={E(s^{Z_t};Z_\infty=0)-P(Z_t=0)\over q-F_t(0)}\\
&={F_t(sq)-F_t(0)\over q-F_t(0)}=1-{q-F_t(sq)\over q-F_t(0)}.
\end{align*}
\end{proof}

%

\begin{theorem}\label{gsup}
 Consider a supercritical case with $1<m<\infty$.  The normalized by its mean branching process converges almost surely 
 $$Z_te^{(1-m)t}\to W,\quad t\to\infty.$$
 If \eqref{xlx} holds, then 
 \begin{align}\label{rol}
 Ee^{-\rho W}=q+(1-q)\phi(\rho),
 \end{align}
where $\phi(\rho)\in(0,1)$, $\rho>0$, satisfies
  \begin{align}\label{ro}
\phi(\rho)=q+(1-q) \Big({1-\phi(\rho)\over\rho}\Big)^{\beta}
\exp\Big\{\int^1_{\phi(\rho)}{\beta\nabla_1^2\nabla_q  f(x)-\nabla_1\nabla_q^{2}  f(x)\over \nabla_1\nabla_q  f(x)}dx\Big\}.
\end{align}
If \eqref{xlx} does not holds, then $P(W=0)=1$.
\end{theorem}
\begin{proof}
Observe that $Z_t/M_t$ forms a non-negative martingale which yields the asserted almost sure convergence and
\begin{equation}\label{cW}
 E(e^{-\rho Z_t/M_t})=F_t(e^{-\rho/M_t})\to Ee^{-\rho W},\quad t\to\infty.
\end{equation}
This martingale property is a corollary of the representation 
\[Z_t=\sum_{k=1}^{Z_u}Z_{u,t}^{(k)},\]
where all $Z_{u,t}^{(k)}$, being mutually independent and independent from the number of summands $Z_u$, have a common distribution
\[Z_{u,t}^{(k)}\stackrel{d}{=}Z_{t-u}.\]

Using \eqref{supeq} with $s=e^{-\rho/M_t}$ we find
\[\gamma^{t}(1-e^{-\rho/M_t})^{-\beta}\to\rho^{-\beta},\quad t\to\infty.\]
By Proposition \ref{pipi}, if \eqref{xlx} holds, then  \eqref{supeq} and \eqref{cW} yield \eqref{rol},
where  $\phi(\rho)$ satisfies \eqref{ro}. On the other hand, if \eqref{xlx} does not hold, then by Proposition  \ref{pipi}, \eqref{supeq}, and \eqref{cW}, we get $F_t(e^{-\rho/M_t})\to1$. 
\end{proof}
\begin{corollary} Consider  the linear-fractional reproduction law \eqref{lf} in the supercritical case. Equation \eqref{ro} takes the form
 \[
\phi(\rho)=q+(1-q) \Big({1-\phi(\rho)\over\rho}\Big)^{1/m}.
\]
If $p_0>0$, then
 \[E(s^{Z_t}|Z_t>0,Z_\infty=0)\to 1-(1-s)(1-sq)^{-1/m}.\]
\end{corollary}

\end{document}